\documentclass[leqno,12pt]{amsart} 

\usepackage{amssymb} 
\usepackage{amsmath}
\usepackage{amsthm}
\usepackage{amscd}
\usepackage{bm}
\usepackage{graphicx,psfrag,wrapfig}
\theoremstyle{plain} 
\newtheorem{theorem}{Theorem}[section] 
\newtheorem{lemma}[theorem]{Lemma}
\newtheorem{corollary}[theorem]{Corollary}
\newtheorem{proposition}[theorem]{Proposition}

\theoremstyle{definition} 

\newtheorem{remark}[theorem]{Remark}
\newtheorem{example}[theorem]{Example}

\newtheorem{problem}[theorem]{Problem}
\newcommand{\affine}{\mathbb{C}}

\newcommand{\norm}[1]{\left|\!\left|#1\right|\!\right|_{\infty}}
\newcommand{\moduli}{\mathcal{M}}

\newcommand{\Widim}{\mathrm{Widim}}
\newcommand{\Diam}{\mathrm{Diam}}
\newcommand{\Alb}{\mathrm{Alb}}

\begin{document}

\title[Moduli space of Brody curves]{Moduli space of Brody curves, energy and mean dimension} 

\author[Masaki Tsukamoto]{Masaki Tsukamoto$^*$} 



\thanks{$^*$Supported by Grant-in-Aid for JSPS Fellows (19$\cdot$1530) from Japan Society for the
Promotion of Science}




\maketitle

\begin{abstract}
We study the mean dimension of the moduli space of Brody curves.
We introduce the notion of ``mean energy" and show that this can be used  
to estimate the mean dimension.
\end{abstract}

\section{Main results}
\subsection{Moduli space of Brody curves}
M. Gromov introduced a remarkable notion of \textit{mean dimension} in \cite{Gromov} 
(see also Lindenstrauss-Weiss \cite{Lindenstrauss-Weiss} and Lindenstrauss \cite{Lindenstrauss}).
In this paper we study the mean dimension of the moduli space of Brody curves.
We introduce the notion of \textit{mean energy} of Brody curves and study the relation 
between mean energy and mean dimension.
Mean energy is, in some sense, an infinite dimensional version of characteristic number, and
our approach is an attempt to attack an infinite dimensional index problem.

Let $\affine P^N$ be the complex projective space and $[z_0: z_1:\cdots :z_N]$ be the 
homogeneous coordinate in $\affine P^N$.
We define the Fubini-Study metric form $\omega_{FS}$ on $\affine P^N$ by 
\begin{equation}\label{def:Fubini-Study}
\omega_{FS} := \frac{\sqrt{-1}}{2 \pi} \partial \bar{\partial}
                \log \left( 1 + \sum_{i = 1}^N |z_i|^2 \right)
\quad \text{on $\{[1:z_1:\cdots:z_N]\}$}.
\end{equation}
This 2-form $\omega_{FS}$ smoothly extends over $\affine P^N$ and defines the Fubini-Study metric. 
This is normalized so that
\[ \int_{\affine P^1} \omega_{FS} = 1
 \quad \text{for $ \, \affine P^1 := \{ \, [z_0: z_1: 0: \cdots : 0] \in \affine P^N \}$ }. \]
Let $z = x + y\sqrt{-1}$ be the natural coordinate in the complex plane $\affine$, and 
let $f:\affine \to \affine P^N$ be a holomorphic map.
We define the pointwise norm $|df|(z)\geq 0$ of the differential $df$ by
\begin{equation}\label{eq: definition of |df|}
 f^* \omega_{FS} = |df|^2 dx dy,
\end{equation}
i.e., for a holomorphic curve $f = [1:f_1:\cdots :f_N]$ 
with holomorphic functions $f_1, \cdots , f_N$
\begin{equation*}
 |df|^2(z) = 2 |df(\partial /\partial z)|^2 = 
 \frac{1}{4 \pi} \Delta \log \left(1 +|f_1|^2+ \cdots +|f_N|^2 \right) \quad 
(\Delta := \frac{\partial^2}{\partial x^2} + \frac{\partial^2}{\partial y^2} ). 
\end{equation*}
We call a holomorphic map $f:\affine \to \affine P^N$ a Brody curve if it satisfies
$|df|\leq 1$ (cf. Brody \cite{Brody}).
Let $\moduli (\affine P^N)$ be the moduli space of Brody curves in $\affine P^N$:
\[
\moduli (\affine P^N) := \{ f: \affine \to \affine P^N | \text{ $f$ is holomorphic and } |df|(z) \leq 1
\text{ for all z $\in \affine$} \}. 
\]
We consider the compact-open topology on $\moduli (\affine P^N)$ 
(in other words, the topology of compact uniform convergence).
This topology is metrizable and $\moduli (\affine P^N)$ becomes a compact topological space.

The Lie group $\affine$ naturally acts on $\moduli (\affine P^N)$:
\[ \affine \times \moduli (\affine P^N) \longrightarrow \moduli (\affine P^N) , \quad (a, f(z)) \mapsto f(z+a). \]
The main objects of study in this paper are \textit{$\affine$-invariant closed subsets in $\moduli (\affine P^N)$}
\footnote{In the theory of dynamical systems, the study of closed invariant sets is very fundamental. 
$\affine$-invariant closed subsets in $\moduli (\affine P^N)$ are their analogue.}.
The following are basic examples:
\begin{example}
Let $X \subset \affine P^N$ be an algebraic set in $\affine P^N$ (not necessarily smooth),
and let $\moduli (X)$ be the moduli space of Brody curves in $X$:
\begin{equation}\label{definiton of M(X)}
\moduli (X) := \{f\in \moduli (\affine P^N)|\, f(\affine ) \subset X \}.
\end{equation}
Since $X$ is closed in $\affine P^N$, this $\moduli (X)$ is closed in $\moduli (\affine P^N)$ and
obviously $\affine$-invariant.
\end{example}
\begin{example}\label{example of quasi-projective cases}
Let $V \subset \affine P^N$ be a hypersurface in $\affine P^N$, i.e. the zero set of a homogeneous polynomial.
Let $\moduli (\affine P^N \setminus  V)$ be the closure of the set of Brody curves in 
$\affine P^N \setminus V$:
\[ \moduli (\affine P^N \setminus  V) := \overline{\{ f\in \moduli (\affine P^N)|\, 
f(\affine ) \subset \affine P^N\setminus V \}}, \]
where the overline means the closure with respect to the compact-open topology.
This becomes a $\affine$-invariant closed subset.
\end{example}


\subsection{Mean dimension and mean energy}
We introduce the notion of \textit{mean energy} in this subsection.
This is the key notion of this paper.
For a holomorphic curve $f:\affine \to \affine P^N$, let $T(r, f)$ be the Shimizu-Ahlfors characteristic function:
\[ T(r, f) := \int_1^r \frac{dt}{t} \int_{|z| \leq t} |df|^2(z) \, dxdy \quad \text{for all $r\geq 1$}. \]
We define the \textit{mean energy} $e(f)$ by 
\begin{equation}\label{def:mean energy}
e(f) := \limsup_{r\to \infty} \frac{2}{\pi r^2} T(r, f).  
\end{equation}
If $f$ is a Brody curve, then we have $T(r, f) \leq \pi r^2/2$. Hence 
\[ 0 \leq e(f) \leq 1 \quad \text{for all $f \in \moduli (\affine P^N)$}. \]
It is easy to see that $e(f)$ is a $\affine$-invariant functional on $\moduli (\affine P^N)$:
\[ e(f(z)) = e(f(z+a)) \quad \text{for any $f(z) \in \moduli (\affine P^N)$ and $a \in \affine$}. \]
For an algebraic set $X\subset \affine P^N$, we define $e(X)$ by 
\[ e(X) := \sup_{f\in \moduli (X)} e(f) .\]
Here $\moduli (X)$ is the moduli space of Brody curves in $X$ defined in (\ref{definiton of M(X)}).
Obviously $e(X)$ satisfies
\[ 0 \leq e(X) \leq 1.\]
\begin{remark}
In \cite{T1}, we introduced and studied the notion of \textit{packing density} of Brody curves.
For a Brody curve $f:\affine \to \affine P^N$, we define the \textit{packing density} $\rho (f)$ by 
\[ \rho (f) := \limsup_{r\to \infty}\frac{1}{\pi r^2}\int_{|z|\leq r} |df|^2(z) \, dxdy.\]
For an algebraic set $X \subset \affine P^N$, we define $\rho (X)$ by 
\[ \rho (X) := \sup_{f\in \moduli (X)} \rho (f) .\]
$\rho (f)$ and $\rho (X)$ obviously satisfies 
\[ 0 \leq \rho (f) \leq 1 \quad \text{and}\quad 0\leq \rho (X) \leq 1 .\]
The integral of $|df|^2(z) dxdy$ is usually called ``energy". Hence $\rho (f)$ measures the 
packing density of the energy of $f$ over the complex plane.

It is easy to see that 
\[ e(f) \leq \rho (f) \quad \text{and hence}\quad e(X) \leq \rho (X) .\]
The crucial point of these notions is the fact that \textit{they are non-trivial invariants}.
In \cite{T1}, the following is proved:
\[ 0 < \rho (\affine P^N) < 1,\]
i.e., the value of $\rho (\affine P^N)$ is non-trivial
\footnote{For the case of $\affine P^1$, we have an effective upper bound (cf. \cite{T1}):
\[ \rho (\affine P^1) \leq 1 - 10^{-100} .\]}.
Hence we can see that\footnote{The upper bound $e(\affine P^N) < 1$ 
follows from $e(\affine P^N) \leq \rho (\affine P^N)$. The lower bound follows from, for example, the fact that 
\[ e(f) > 0 \quad \text{for a non-constant elliptic function $f:\affine \to \affine P^1$}.\]  } 
\begin{equation*}
 0 < e(\affine P^N) < 1 .
\end{equation*}
In particular we have
\[ e(X) \leq e(\affine P^N) < 1 \quad \text{for any algebraic set $X$ in $\affine P^N$} .\]

If $f:\affine \to \affine P^N$ extends to a holomorphic map $\tilde{f}$ from $\affine P^1 = \affine \cup \{\infty\}$ to 
$\affine P^N$, then the total energy is equal to the \textit{degree} of the map $\Tilde{f}$:
\[ \int_{\affine} |df|^2 \, dxdy = \int_{\affine P^1} \tilde{f}^* \omega_{FS} = \deg(\tilde{f}) .\]
Therefore $e(f)$ and $\rho (f)$ are ``regularized degree" of Brody curves.
\end{remark}

For an algebraic set $X\subset \affine P^N$, $\moduli (X)$ is a compact topological space 
whose topology is metrizable, and the Lie group $\affine$ acts on $\moduli (X)$.
 Then we can consider the mean dimension $\dim (\moduli (X):\affine )$ 
(cf. Gromov \cite{Gromov} and Subsection 4.1).
The mean energy $e(X)$ gives an upper bound for $\dim (\moduli (X):\affine )$:
\begin{theorem}\label{theorem: upper bound for dim(M(X):C)}
\[ \dim(\moduli (X) :\affine ) \leq 4 e(X) \dim_{\affine}X .\]
Here $\dim_{\affine}X$ denotes the complex dimension of $X$.
For the definition of complex dimension of algebraic sets, see Grauert-Remmert \cite[Chapter 5]{Grauert}.
\end{theorem}
This result is a start point of the study of the relation between mean dimension and mean energy.


\subsection{The case of $\affine P^N$}
Applying Theorem \ref{theorem: upper bound for dim(M(X):C)} to the case of $X = \affine P^N$, 
we get an upper bound:
\begin{equation}\label{upper bound for dim(M(CP^N):C)}
 \dim (\moduli(\affine P^N) :\affine ) \leq 4 e(\affine P^N) N < 4N .
\end{equation}
Here we have used the fact $e(\affine P^N) < 1$.
On the other hand, from Gromov \cite[p. 328, 0.6.2]{Gromov}, 
we have $\dim(\moduli (\affine P^N) :\affine ) >0$. 
Actually we can prove
\begin{theorem}\label{theorem: lower bound for the case of projective space}
There exists a positive constant $C$ independent of $N$ such that 
\[ \dim (\moduli (\affine P^N) :\affine ) \geq C \cdot N.\]
Therefore
\[ C\cdot N \leq \dim (\moduli (\affine P^N):\affine ) \leq 4 e(\affine P^N) N < 4N .\]
\end{theorem}
\begin{remark}
Gromov gives a certain upper bound for $\dim(\moduli(\affine P^N):\mathbb{\affine})$ in 
\cite[p. 396, (c)]{Gromov}.
Unfortunately, I could not find the definition of 
the Fubini-Study metric used in \cite[p. 396, (c)]{Gromov} (the Fubini-Study metric has 
several conventions).
Therefore I could not decide whether our estimate (\ref{upper bound for dim(M(CP^N):C)}) 
is better than Gromov's estimate in \cite[p. 396, (c)]{Gromov} or not.
But Gromov referred to the paper of A. Eremenko \cite{Eremenko} there, and the argument in 
\cite[Theorem 2.5]{Eremenko} is similar to our argument in Lemma 2.1 of this paper. 
And I think that the use of mean energy (or packing density) makes the related estimates
sharper.
\end{remark}
\begin{problem}
In \cite[Section 4]{T1} we proved that 
\[ \lim_{N\to \infty} \rho(\affine P^N) = 1.\]
In the same way as in \cite[Section 4]{T1}, we can prove
\[ \lim_{N\to \infty} e(\affine P^N) = 1 .\]
Hence it might be interesting to study the asymptotic behavior of 
$\dim(\moduli(\affine P^N):\affine)/N$ as $N\to \infty$.
\end{problem}


\subsection{Definition of $\moduli_+$ and some examples of mean dimension $= 0$}
Let $\moduli \subset \moduli (\affine P^N)$ be a $\affine$-invariant closed subset in $\moduli (\affine P^N)$.
We define $\moduli_+\subset \moduli$ as the closure of the set of Brody curves in $\moduli$ of positive mean energy:
\[ \moduli_+ := \overline{\{ f \in \moduli |\, e(f) >0\}}. \]
$\moduli_+$ is a $\affine$-invariant closed subset in $\moduli$.
We have the following general fact:
\begin{theorem} \label{theorem: dim(M:C) = dim(M_+:C)}
\[ \dim (\moduli :\affine ) = \dim (\moduli_+: \affine ) .\]
If $\moduli_+$ is empty, then we set $\dim(\moduli_+ :\mathbb{C}) = \dim (\emptyset:\affine ) :=0$.
\end{theorem}
If $\moduli_+$ is a finite dimensional space (in the sense of topological covering dimension),
 then we have $\dim (\moduli_+ :\affine ) =0$.
Therefore 
\begin{corollary}
If $\dim (\moduli :\affine)$ is positive, then $\moduli_+$ is an infinite dimensional space.
\end{corollary}
The following are examples of spaces whose mean dimension is $0$. 
\begin{example}
This is a trivial example.
Let $X\subset \affine P^N$ be a compact hyperbolic manifold, i.e., all holomorphic curves in $X$ are constant maps.
Then $\moduli (X)$ consists of constant maps, and it is homeomorphic to $X$. $\moduli(X)_+$ is empty 
and we have
\[ \dim (\moduli (X) :\affine ) = \dim (\moduli(X)_+ :\affine ) =0 .\]
\end{example}
\begin{example}
Let $H_0, H_1, \cdots , H_{N}$ be the $N+1$ hyperplanes in $\affine P^N$:
\[
 H_i:\quad  \sum_{j = 0}^N a_{i j} z_j  = 0 \quad (0 \leq i \leq N).  
\]
Suppose that $H_0, H_1, \cdots , H_{N}$ are linearly independent, 
i.e., the coefficients matrix $(a_{ij})_{0 \leq i, j \leq N}$ is regular.
Let $\moduli$ be the closure of the set of Brody curves contained in 
$\affine P^N \setminus (H_0\cup \cdots \cup H_N)$:
\[ \moduli := \moduli (\affine P^N \setminus (H_0\cup \cdots \cup H_N)) = 
\overline{\{f\in \moduli (\affine P^N)|\, f(\affine ) \subset \affine P^N \setminus (H_0\cup \cdots \cup H_N)\}}
.\]
Then $\moduli$ is a finite dimensional space and $\moduli_+$ is empty. In particular we have
\[ \dim (\moduli :\affine ) = \dim (\moduli_+ :\affine ) = 0 .\]
\begin{proof}
We consider only the case of $H_i = \{z_i = 0\}$ $(0\leq i\leq N)$ for simplicity.
If $f$ is contained in $\moduli$, then we have
\[ f(\affine) \cap H_i = \emptyset \quad \text{or} \quad  f(\affine ) \subset H_i \quad \text{for each $H_i$}.\]
Suppose that 
\[ f(\affine ) \cap H_i = \emptyset \text{ for $i = 0,1,\cdots, m$, and } 
f(\affine ) \subset H_j  \text{ for $j = m+1, m+2, \cdots, N$}.\]
Then $f = [g: 0:\cdots:0]$ with a Brody curve $g:\affine \to \affine P^m\setminus (H_0 \cup H_1\cup \cdots \cup H_m)$.
From the theorem of F. Berteloot and J. Duval in \cite[Appendice]{Berteloot} (see also \cite[Section 6]{T1} and \cite{T2}), 
the Brody curve $g$ can be expressed by 
\[ g(z) = [1: e^{a_1 z+ b_1}: \cdots :e^{a_m z + b_m}] \quad \text{for some complex numbers $a_1, b_1, \cdots, a_m, b_m$} \]
Hence all $f\in \moduli $ can be expressed by (cf. \cite[Section 3]{Berteloot})
\[ f(z) = [c_0 e^{a_0 z}: c_1 e^{a_1 z}:\cdots :c_N e^{a_N z}] \quad 
\text{for some complex numbers $a_0, c_0, \cdots, a_N, c_N$}.\]
In addition we have $e(f) = 0$.
Therefore
\[ \dim (\moduli ) = 4N \quad \text{and}\quad \moduli_+ = \emptyset.\]
\end{proof}
\end{example}
\begin{example}
\[ \moduli := \moduli (\affine P^1 \setminus \{\infty\}) 
= \{ f\in \moduli (\affine P^1) |\, f(\affine) \subset \affine = \affine P^1 \setminus \{\infty \}
 \text{ or } f(\affine) = \{\infty \} \} .\]
For any polynomial $p(z)$, if we choose $\varepsilon >0$ sufficiently small, $p(\varepsilon z)$ belongs to $\moduli$.
Then it is easy to see that 
$\moduli$ is an infinite dimensional space. But from Minda \cite[Theorem 5]{Minda}, 
all $f\in \moduli$ has order $\leq 1$:
\[ \limsup_{r\to \infty} \log T(r, f) /\log r  \leq 1 .\]
Hence $e(f) = 0$ for all $f\in \moduli$ and $\moduli_+ = \emptyset$. Therefore 
\[ \dim(\moduli :\affine ) = \dim (\moduli_+ :\affine) = 0,\]
i.e., $\moduli$ is an infinite dimensional space whose mean dimension is $0$.
\end{example}
\begin{problem}
I don't know whether there is a hypersurface $V \subset \affine P^N$ such that 
\[ \dim (\moduli(\affine P^N \setminus V) :\affine ) >0,\]
(cf. Example \ref{example of quasi-projective cases}).
I don't know even whether there is a hypersurface $V$ such that 
$\moduli(\affine P^N\setminus V)_+$ is not empty.
\end{problem}


\subsection{Holomorphic 1-forms and mean dimension}
Theorem \ref{theorem: upper bound for dim(M(X):C)} can be applied to general algebraic sets.
If $X\subset \affine P^N$ is smooth, connected and has holomorphic 1-forms, 
then we can improve the estimate.

Let $X\subset \affine P^N$ be a smooth connected projective variety, and let $H^{1,0}$ be the space of holomorphic 
1-forms on $X$.
Let $\omega_1, \cdots, \omega_h$ ($h= \dim_{\affine} H^{1,0}$) be the basis of $H^{1,0}$, and let 
$\alpha$ be the Albanese map:
\[ \alpha: X \to \Alb(X) := \affine^h /\Gamma ,\quad 
x \mapsto (\int_p^x \omega_1, \cdots, \int_p^x \omega_h) ,\]
where $p\in X$ is a reference point and $\Gamma \subset \affine ^h$ is the lattice given by periods:
\[ \Gamma := \{ (\int_C \omega_1, \cdots, \int_C \omega_h )\in \affine^h|\, C \in H_1 (X:\mathbb{Z})\} .\]
The important data for us is the derivative $d\alpha$ of the Albanese map $\alpha$:
\[ d\alpha_x :T_xX = T_x^{1,0}X \to \affine^h, \quad v\mapsto (\omega_1(v), \cdots, \omega_h(v))   
\quad \text{for $x\in X$}.\]
We define the closed analytic set $Y\subset X$ by 
\begin{equation}\label{definition of Y}
Y := \{ x\in X|\, d\alpha_x \text{ is not injective}\}.
\end{equation}
Since $X\subset \affine P^N$, $Y$ is also an algebraic set in $\affine P^N$.
\begin{theorem}\footnote{Here we assume that $X$ is projective.
But actually this theorem is valid for compact K\"{a}hler manifolds. See Section 5.}
\label{theorem: dim(M(X):C) = dim(M(Y):C)}
\[ \dim(\moduli (X):\affine ) = \dim(\moduli (Y):\affine ) .\]
\end{theorem}
\begin{example}
Let $X$ be a compact smooth algebraic curve of genus $\geq 1$.
It is well-known that the Albanese map $\alpha: X\to \Alb(X)$ becomes an embedding.
Hence $Y = \emptyset$ and 
\[ \dim (\moduli (X) :\affine ) = 0.\]
(In fact $\moduli (X)$ is a finite dimensional space in this case.)
\end{example}
If $h = \dim_{\affine} H^{1,0} < \dim_\affine X$, then $Y=X$ and 
Theorem \ref{theorem: dim(M(X):C) = dim(M(Y):C)} becomes meaningless. 
Next we will develop a theorem which can cover this case.
The map 
\[ d\alpha :TX \to \affine^h, \quad v \mapsto (\omega_1 (v) , \cdots , \omega_h (v) ) \]
is a holomorphic map. Hence for each $u\in \affine^h$, the inverse image
$(d\alpha)^{-1} (u) \subset TX$ is a closed analytic set in $TX$.
If $h>0$, then $(d\alpha)^{-1} (u)$ is nowhere dense in $TX$ (by connectedness of $X$), and we have
\[ \dim_\affine \,(d\alpha)^{-1} (u) < \dim_\affine TX = 2\dim_\affine X \quad \text{for all $u\in \affine^h$}. \]
This dimension $\dim_\affine \,(d\alpha)^{-1} (u)$ can be used for an upper bound of the mean dimension:
\begin{theorem}\label{theorem: holomorphic 1-forms and mean dimension}
If $X$ is a smooth connected projective variety, then 
\[ \dim(\moduli (X):\affine ) \leq 2e(X) \max_{u\in \affine^h}\{ \dim_\affine \,(d\alpha)^{-1} (u)\}.\]
In particular, if $h>0$, then 
\[\dim(\moduli (X) :\affine) \leq e(X) (4\dim_\affine X -2) .\]
\end{theorem} 


\subsection{Organization of the paper}
In Section 2 we study ``discretization" of holomorphic curves and prove 
Theorem \ref{theorem: upper bound for dim(M(X):C)}.
In Section 3 we prove Theorem \ref{theorem: lower bound for the case of projective space}.
In Subsection 4.1 we review the definitions and basic properties of mean dimension.
Readers who are not familiar with mean dimension can read Subsection 4.1 first 
before reading other sections.
In Subsection 4.2 we show some general results on mean dimension and prove 
Theorem \ref{theorem: dim(M:C) = dim(M_+:C)}.
In Section 5 we prove Theorem \ref{theorem: dim(M(X):C) = dim(M(Y):C)} and
\ref{theorem: holomorphic 1-forms and mean dimension}.


\section{Discretization of holomorphic curves}\label{section: Discretization of holomorphic curves}
Let $\Lambda \subset \affine$ be a lattice in the complex plane $\affine$.
Let $(\affine P^N)^\Lambda$ be the infinite product of the copies of $\affine P^N$ indexed by $\Lambda$:
\[ (\affine P^N)^\Lambda := \{ (w_\lambda)_{\lambda \in \Lambda}|\, w_\lambda \in \affine P^N\} .\]
First we study the following ``discretization map" (cf. Gromov \cite[p. 329]{Gromov}):
\[ \moduli (\affine P^N) \to (\affine P^N)^\Lambda , 
\quad f\mapsto f|_{\Lambda} := (f(\lambda ))_{\lambda \in \Lambda} .\]
\begin{lemma}\label{lemma: discretization 1}
Let $f, g:\affine \to \affine P^N$ be holomorphic maps with $e(f), e(g) < \infty$, and suppose 
\[ e(f) + e(g) < 1/|\affine /\Lambda | .\]
Here $|\affine /\Lambda |$ is the volume of the elliptic curve $\affine /\Lambda$, 
i.e. the area of the fundamental domain of $\Lambda$ in $\affine$.
Then if $f|_{\Lambda} = g|_{\Lambda}$, we have $f \equiv g$.
\end{lemma}
\begin{proof}\footnote{A similar argument is given by A. Eremenko in \cite[Theorem 2.5]{Eremenko}.}
Let $[z_0:\cdots :z_n]$ be the homogeneous coordinate in $\affine P^N$.
Since $f(\Lambda) = g(\Lambda)$ is a countable set in $\affine P^N$, there is a hyperplane
$H\subset \affine P^N$ such that $f(\Lambda)\cap H = \emptyset$ (by Baire's theorem).
We can suppose that $H = \{z_0 =0\}$ without loss of generality.

Then we can express $f$ and $g$ by $f=[1:f_1:\cdots:f_N]$ and $g=[1:g_1:\cdots:g_N]$ with 
meromorphic functions $f_1, \cdots, f_N, g_1\cdots, g_N$ such that $f_i|_{\Lambda} = g_i|_{\Lambda}$ and 
$\infty \notin f_i(\Lambda) = g_i(\Lambda)$. 
The standard argument in the Nevanlinna theory gives
\[ T(r, f_i) \leq T(r, f) + O(1) .\]
Hence we have $e(f_i) \leq e(f)$ and $e(g_i) \leq e(g)$.
We want to prove that $f_i \equiv g_i$ for all $i$.
Suppose $f_1 \not\equiv g_1$. Then non-constant meromorphic function $(f_1 - g_1)^{-1}$ has a pole
at each point of the lattice $\Lambda$. (Here we have used $\infty \notin f_i(\Lambda) = g_i(\Lambda)$.)
From the first main theorem of Nevanlinna,
\[\frac{\pi r^2}{2|\affine/\Lambda|} + O(r) \leq T(r, (f_1 - g_1)^{-1}) = T(r, f_1 - g_1) + O(1) 
\leq T(r, f_1) + T(r, g_1) + O(1) .\]
Then 
\[ \frac{1}{|\affine /\Lambda|} \leq e(f_1) + e(g_1) \leq e(f) + e(g) .\]
This contradicts the assumption.
\end{proof}
\begin{remark}
In the above proof, we did not use the second main theorem of Nevanlinna.
Actually we don't need the second main theorem in any part of this paper.
I don't know how to apply the second main theorem to the theory of mean dimension.
\end{remark}
Next we study the following map:
\[ \moduli (\affine P^N) \to (T\affine P^N)^\Lambda ,\quad
 f\mapsto df|_{\Lambda} := (df (\partial /\partial z)|_{z = \lambda})_{\lambda \in \Lambda}
\in \prod_{\lambda \in \Lambda} T_{f(\lambda)}\affine P^N .\]
This map has the information of derivative of holomorphic curves at each point of the lattice $\Lambda$.
\begin{lemma}\label{lemma: discretization 2}
Let $f, g:\affine \to \affine P^N$ be holomorphic maps with $e(f), e(g) < \infty$, and suppose 
\[ e(f) + e(g) < 2/|\affine /\Lambda | .\]
Then if $df|_{\Lambda} = dg|_{\Lambda}$, we have $f \equiv g$.
\end{lemma}
\begin{proof}
$df|_{\Lambda} = dg|_{\Lambda}$ implies $f|_{\Lambda} = g|_{\Lambda}$ by definition.
Hence we can suppose that we can express $f$ and $g$ by $f=[1:f_1:\cdots:f_N]$ and $g=[1:g_1:\cdots:g_N]$ with 
meromorphic functions $f_1, \cdots, f_N, g_1\cdots, g_N$ such that $f_i|_{\Lambda} = g_i|_{\Lambda}$ and 
$\infty \notin f_i(\Lambda) = g_i(\Lambda)$. We have also $e(f_i) \leq e(f)$ and $e(g_i) \leq e(g)$.
From $df|_{\Lambda} = dg|_{\Lambda}$, we have 
$f_i(\lambda) = g_i(\lambda)$ and $f'_i(\lambda) = g'_i(\lambda)$ for all $\lambda \in \Lambda$. 
Suppose $f_1 \not\equiv g_1$.
Then the non-constant meromorphic function $(f_1 -g_1)^{-1}$ has 
a pole of multiplicity $\geq 2$ at each point of $\Lambda$.
From the first main theorem of Nevanlinna,
\[\frac{\pi r^2}{|\affine /\Lambda |} + O(r) \leq T(r, (f_1 - g_1)^{-1}) =  T(r, f_1 - g_1) + O(1) 
\leq T(r, f_1) + T(r, g_1) + O(1) .\]
Then
\[ \frac{2}{|\affine /\Lambda|} \leq e(f_1) + e(g_1) \leq e(f) + e(g) .\]
This contradicts the assumption.
\end{proof}
We prove Theorem \ref{theorem: upper bound for dim(M(X):C)} by using Lemma \ref{lemma: discretization 1}.
(Lemma \ref{lemma: discretization 2} will be used later in the proof of 
Theorem \ref{theorem: holomorphic 1-forms and mean dimension}.)
\begin{proof}[Proof of Theorem \ref{theorem: upper bound for dim(M(X):C)}]
Let $\Lambda \subset \affine$ be a lattice satisfying 
\[ 2e(X) < 1/|\affine/\Lambda| \]
where $X\subset \affine P^N$ is a given algebraic set.
Consider the ``discretization map":
\begin{equation}\label{eq: discretization for X}
 \moduli (X) \to X^\Lambda, \quad f\mapsto f|_{\Lambda} .
\end{equation}
This map naturally becomes a $\Lambda$-equivariant map,
and it is continuous (here we consider the direct product topology on $X^\Lambda$).
For any two $f, g\in \moduli (X)$ we have
\[ e(f) + e(g) \leq 2e(X) < 1/|\affine/\Lambda|. \]
Then Lemma \ref{lemma: discretization 1} implies that 
the discretization map (\ref{eq: discretization for X}) is injective.
Since $\moduli (X)$ is compact, this means that $\moduli (X)$ is $\Lambda$-equivariantly
homeomorphic to the image of (\ref{eq: discretization for X}). Therefore
\[ \dim(\moduli (X):\Lambda) \leq \dim(X^\Lambda :\Lambda ) = \dim X = 2\dim_\affine X \]
where $\dim X$ denotes the topological covering dimension of $X$.
Then (cf. Subsection 4.1)
\[ \dim (\moduli (X):\affine) = |\affine /\Lambda|^{-1} \dim(\moduli (X) :\Lambda) 
\leq 2|\affine /\Lambda |^{-1} \dim_\affine X .\]
$|\affine /\Lambda |^{-1}$ can be taken arbitrarily close to $2e(X)$. Hence 
\[ \dim(\moduli (X):\affine ) \leq 4e(X)\dim_\affine X.\]
\end{proof}


\section{Constructing a shift space in $\moduli (\affine P^N)$}
To begin with, note that the following map is a holomorphic isometric embedding:
\[ \affine P^1 \to \affine P^N , \quad [1:z]\to [1:z/\sqrt{N}:\cdots:z/\sqrt{N}] .\]
This fact is behind the arguments in this section.

In this section we will prove Theorem \ref{theorem: lower bound for the case of projective space}
by constructing a ``shift space" in $\moduli (\affine P^N)$. 
Our argument is a variant of the argument of Gromov \cite[p. 398, 3.5.1]{Gromov}.
Let $\Lambda\subset \affine$ be a lattice, $A>0$ be a positive number.
We define the annulus $\Omega \subset \affine$ by
\[ \Omega := \{z\in \affine |\, A\leq |z|\leq 2A \} .\]
For $a = (a_{n\lambda})_{1\leq n\leq N, \lambda\in \Lambda} \in (\Omega^N)^\Lambda$ 
(i.e. $A\leq |a_{n\lambda}|\leq 2A$), we define the holomorphic map $f_a:\affine \to \affine P^N$ by
\[ f_a(z):= [1: \frac{1}{\sqrt{N}}\sum_{\lambda\in \Lambda} \frac{a_{1\lambda}}{(z-\lambda)^3} :
\frac{1}{\sqrt{N}}\sum_{\lambda\in \Lambda} \frac{a_{2\lambda}}{(z-\lambda)^3}:\cdots :
\frac{1}{\sqrt{N}}\sum_{\lambda\in \Lambda} \frac{a_{N\lambda}}{(z-\lambda)^3}] .\]
The following is the basis of the proof:
\begin{proposition}\label{proposition: |df_a| < const}
There is a positive constant $C(\Lambda, A)$ independent of $N$ such that 
\[ |df_a|(z) \leq C(\Lambda, A) \quad \text{for all $z\in \affine$ and all $a\in (\Omega^N)^\Lambda$}.\]
(The important point of this statement is that $C(\Lambda, A)$ is independent of $N$.)
\end{proposition}
Let $\delta = \delta(\Lambda)$ be a positive number satisfying
\[  2\delta \leq |\lambda_1 -\lambda_2| \quad \text{for any $\lambda_1, \lambda_2 \in \Lambda$ with
$\lambda_1 \neq \lambda_2$} .\]
The proof of Proposition \ref{proposition: |df_a| < const} needs the following lemma
(similar estimates are given in Eremenko \cite[Lemma 6.2]{Eremenko}):
\begin{lemma}\label{lemma: estimate of infinite sum}
For positive numbers $s>2$ and $d\leq \delta$, there is a positive constant $c_1(\Lambda, s, d)$ 
such that
\[ \sum_{\lambda\in \Lambda} \frac{1}{|z-\lambda|^s} \leq c_1(\Lambda, s, d) \quad 
\text{for all $z\in \affine$ with $d(z, \Lambda) \geq d$} .\]
Moreover there is a positive constant $c_2(\Lambda, s)$ satisfying the following;
for any $z\in \affine$ with $d(z, \Lambda) < \delta$, 
let $\lambda_0\in \Lambda$ be a (unique) point in $\Lambda$ such that $|z-\lambda_0|\ <\delta$. 
Then
\[ \sum_{\lambda\in \Lambda\setminus\{\lambda_0\}} \frac{1}{|z-\lambda|^s} \leq c_2(\Lambda, s). \]
\end{lemma}
\begin{proof}
We can prove these results by direct estimation; we omit the detail.
\end{proof}
\begin{proof}[Proof of Proposition \ref{proposition: |df_a| < const}]
Set $f_a(z) = [1:f_1(z):f_2(z): \cdots :f_N(z)]$. From (\ref{eq: definition of |df|}), we have
\begin{equation*}
\pi |df_a|^2 = \frac{\sum |f_n'|^2 + \sum_{n<m} |f_n f_m' - f_n' f_m|^2}{(1+\sum |f_n|^2)^2}
\end{equation*}
Set 
\[ d := \min \left(\delta ,  \left(\frac{1}{4c_2(\Lambda, 3)}\right)^{1/3} \right) .\]
Suppose $z\in \affine$ satisfies $d(z, \Lambda) \geq d$. 
Then Lemma \ref{lemma: estimate of infinite sum} shows
\[ |f_n'(z)| \leq \frac{1}{\sqrt{N}}\sum_{\lambda\in\Lambda} \frac{6A}{|z-\lambda|^4}
\leq \frac{6A}{\sqrt{N}}c_1(\Lambda, 4, d) .\]
From this, we have
\[ \sum_n |f_n'(z)|^2 \leq \mathrm{const}_{\Lambda, A} .\]
Here $\mathrm{const}_{\Lambda, A}$ denotes a positive constant independent of $N$
(and depending on $\Lambda, A$).
In the same way, we have
\[ \sum_{n<m} |f_n' f_m - f_m' f_n|^2 \leq \mathrm{const}_{\Lambda, A} .\]
Hence 
\[ |df_a|(z) \leq \mathrm{const}_{\Lambda, A} \quad 
\text{for $z\in \affine$ with $d(z, \Lambda) \geq d$} .\]

Next suppose $z \in \affine$ satisfies $d(z, \Lambda) < d$. 
From $d\leq \delta$, there is a unique $\lambda_0\in \Lambda$ such that
$|z-\lambda_0| < d$. We suppose $\lambda_0 =0$ for simplicity, i.e. $|z|< d$.
We have
\[ f_a(z) = [z^3: \frac{a_{1 0}}{\sqrt{N}} + 
\frac{z^3}{\sqrt{N}} \sum' \frac{a_{1 \lambda}}{(z-\lambda)^3} : \cdots : 
\frac{a_{N 0}}{\sqrt{N}} + \frac{z^3}{\sqrt{N}} \sum' \frac{a_{N \lambda}}{(z-\lambda)^3}] .\]
where $\displaystyle\sum'$ denotes the sum over $\lambda \in \Lambda\setminus\{0\}$.
Set 
\[ g_n(z) := z^3 f_n(z) =
\frac{a_{n 0}}{\sqrt{N}} + \frac{z^3}{\sqrt{N}} \sum' \frac{a_{n \lambda}}{(z-\lambda)^3} .\]
Then we have
\[ \pi |df_a|^2(z) = \frac{\sum |z^3 g_n' - 3z^2 g_n|^2 + 
\sum_{n<m}|g_n g_m' -g_n' g_m|^2}{(|z|^6 + \sum |g_n|^2)^2} .\]
From $|z| < d$, $d^3 \leq 1/4c_2(\Lambda, 3)$ and Lemma \ref{lemma: estimate of infinite sum},
we have
\[|g_n(z)| \geq \frac{A}{\sqrt{N}} - \frac{d^3}{\sqrt{N}}\sum' \frac{2A}{|z-\lambda|^3}
\geq \frac{A}{\sqrt{N}} - \frac{2Ad^3}{\sqrt{N}}c_2(\Lambda, 3) \geq \frac{A}{2\sqrt{N}} .\]
Hence
\[ \sum |g_n(z)|^2 \geq \frac{A^2}{4} .\]
Therefore
\[ \pi |df_a|^2(z) \leq \frac{16}{A^4}
\left[\sum |z^3 g_n' - 3z^2 g_n|^2 + \sum_{n<m}|g_n g_m' -g_n' g_m|^2 \right] .\]
Then some calculation shows
\[ |df_a|(z) \leq \mathrm{const}_{\Lambda, A} \quad \text{for $|z|<d$} .\]
Thus we conclude that 
\[ |df_a|(z) \leq C(\Lambda, A) \quad \text{for all $z\in \affine$} .\]
\end{proof}
If we set $\Hat{f}_a(z) := f_a(z/c(\Lambda, A))$, then we have $|d\Hat{f}_a|(z)\leq 1$.
Therefore we get the following:
\begin{corollary}
There are $\Lambda$ and $A$ independent of $N$ such that
\[ |df_a| \leq 1 \quad \text{for all $a\in (\Omega^N)^\Lambda$}.\]
\end{corollary}
Hence we get the following map:
\[ F: (\Omega^N)^\Lambda \to \moduli (\affine P^N), \quad a\mapsto f_a .\]
$F$ is obviously injective and $\Lambda$-equivariant.
Moreover some calculation shows that $F$ is continuous (here we consider the product topology
on $(\Omega^N)^\Lambda$ and the compact-open topology on $\moduli(\affine P^N)$).
Hence $F$ is a topological embedding.
\begin{proof}[Proof of Theorem \ref{theorem: lower bound for the case of projective space}]
$\moduli(\affine P^N)$ contains a ``shift space" $(\Omega^N)^\Lambda$. Hence
\[ \dim(\moduli (\affine P^N): \Lambda) \geq \dim((\Omega^N)^\Lambda :\Lambda) = 2N. \]
Therefore (cf. Subsection 4.1)
\[ \dim(\moduli(\affine P^N):\affine ) = \dim(\moduli (\affine P^N): \Lambda) /|\affine /\Lambda|
\geq 2N /|\affine /\Lambda| .\]
Note that $\Lambda$ is independent of $N$. Hence this shows the theorem.
\end{proof}


\section{General theory of mean dimension} \label{section: General theory of mean dimension}
\subsection{Review of mean dimension}\label{subsection: Review of mean dimension}
We review the basic definitions of mean dimension given in Gromov \cite{Gromov} 
(see also Lindenstrauss-Weiss \cite{Lindenstrauss-Weiss}).
All results in this subsection are essentially given in \cite{Gromov, Lindenstrauss-Weiss}.
Let $(X, d)$ be a compact metric space and $Y$ a topological space.
For a positive number $\varepsilon >0$, a continuous map $f:X\to Y$ is called an 
$\varepsilon$-embedding if we have $\Diam (f^{-1}(y)) \leq \varepsilon$ for any point $y\in Y$.
We define $\Widim_\varepsilon (X, d)$ as the minimum number $n$ such that 
there exist an $n$-dimensional finite polyhedron $P$ and an $\varepsilon$-embedding $f:X\to P$.
Since $X$ is compact, $\Widim_\varepsilon (X, d) < \infty$. 
$\Widim_\varepsilon (X, d)$ is monotone non-decreasing as $\varepsilon \to 0$,
and we have
\[ \lim_{\varepsilon \downarrow 0} \Widim_\varepsilon (X, d) = \dim X \]
where $\dim X$ denotes the topological covering dimension of $X$ 
(of course $\dim X$ can be infinite).
The following is a fundamental example (see \cite[pp. 332-333]{Gromov} and \cite[Lemma 3.2]{Lindenstrauss-Weiss}).
\begin{example}\label{example: fundamental exmaple of Widim}
Let $d_\infty(\cdot, \cdot)$ be the sup-distance on $[0,1]^N$: $d_\infty(x, y)= \max_{i} |x_i -y_i|$.
Then
\[ \Widim_\varepsilon ([0,1]^N, d_\infty) =N  \quad \text{for any $\varepsilon <1$}.\]
The important point of this statement is that the estimate $\varepsilon <1$ is independent of $N$.
\end{example}
\begin{proof}
$[0,1]^N$ is itself a finite polyhedron of dimension $N$. Hence $ \Widim_\varepsilon ([0,1]^N, d_\infty) \leq N$ 
is obvious. Consider the constant sheaf $\mathbb{Z}$ on $[0,1]^N$, 
and we define the subsheaf $\mathcal{F}\subset \mathbb{Z}$ by
\[ \mathcal{F}_p = \mathbb{Z}_p \quad \text{for $p\in (0,1)^N$} \quad \text{and}\quad 
\mathcal{F}_p = 0 \quad \text{for $p\in \partial [0,1]^N$} .\] 
The \v{C}ech cohomology $\Check{H}^*([0,1]^N, \mathcal{F})$ is equal to the cohomology
$H^*([0, 1]^N, \partial [0,1]^N)$.
In particular 
\[ \Check{H}^N([0,1]^N, \mathcal{F}) = \mathbb{Z}.\]
Set $U_0 := [0,1)$ and $U_1 := (0,1]$, and we define the open covering $\mathcal{U} =\{U_{i_1 \cdots i_N}\}$ 
of $[0,1]^N$ by
\[ U_{i_1\cdots i_N} := U_{i_1}\times \cdots \times U_{i_N} \quad \text{for all $i_1, \cdots, i_N =0, 1$} .\]
$\mathcal{U}$ is acyclic for $\mathcal{F}$, and hence 
the natural map $\Check{H}^* (\mathcal{U}, \mathcal{F}) \to \Check{H}^*([0,1]^N, \mathcal{F})$ is isomorphic
(by Leray's theorem).

Suppose $\Widim_\varepsilon ([0,1]^N, d_\infty) \leq N-1$ for some $\varepsilon <1$.
Then there exists a open covering $\mathcal{V}$ of $[0,1]^N$ such that $\mathcal{V}$ is a refinement of $\mathcal{U}$ and
the order of $\mathcal{V}$ is $\leq N-1$, i.e., any intersection of $N+1$ open sets in $\mathcal{V}$ is empty.
Then the isomorphism $\Check{H}^N (\mathcal{U}, \mathcal{F}) \to \Check{H}^N([0,1]^N, \mathcal{F}) = \mathbb{Z}$ is equal to
the zero map:
\[ \Check{H}^N (\mathcal{U}, \mathcal{F}) \to \Check{H}^N(\mathcal{V}, \mathcal{F}) =0 \to \Check{H}^N([0,1]^N, \mathcal{F}).\]
This is a contradiction.
\end{proof}

Suppose that the additive group $\mathbb{Z}^k$ $(k\geq 1)$ acts on $X$.
For a finite subset $\Omega$ in $\mathbb{Z}^k$, we define the distance $d_\Omega(\cdot, \cdot)$ on $X$ by 
\begin{equation}\label{eq: definition of d_omega}
 d_\Omega (x, y) := \max_{\gamma \in \Omega} d(\gamma .x, \gamma .y) \quad \text{for $x, y \in X$}.
\end{equation}
$(X, d_\Omega)$ is homeomorphic to $(X, d)$. 
In particular, $(X, d_\Omega)$ is compact and $\Widim_\varepsilon(X, d_\Omega)$ can be defined.
For a positive integer $n$, we set $I_n := [0, n)^k \cap \mathbb{Z}^k$. 
The sequence $\{I_n\}_{n\geq 1}$ is amenable in $\mathbb{Z}^k$ (in the sense of \cite[p. 335]{Gromov}), 
and we can define $\Widim_\varepsilon (X: \mathbb{Z}^k)$ by
\[ \Widim_\varepsilon (X:\mathbb{Z}^k) := \lim_{n\to \infty}\frac{1}{n^k} \Widim_\varepsilon (X, d_{I_n}).\]
This limit always exists; see \cite[pp. 335-338]{Gromov} and \cite[Appendix]{Lindenstrauss-Weiss}.
$\Widim_\varepsilon (X:\mathbb{Z}^k)$ is monotone non-decreasing as $\varepsilon \to 0$,
and we define the mean dimension $\dim (X:\mathbb{Z}^k)$ by
\[ \dim (X: \mathbb{Z}^k) := \lim_{\varepsilon \downarrow 0} \Widim_\varepsilon (X:\mathbb{Z}^k) .\]
If $\dim X <\infty$, then $\Widim_\varepsilon (X:d_{I_n} ) \leq \dim X <\infty$ and 
$\Widim_\varepsilon(X:\mathbb{Z}^k) =0$. 
In particular (cf. Lindenstrauss-Weiss \cite[p. 6]{Lindenstrauss-Weiss})
\[ \dim(X:\mathbb{Z}^k ) = 0 \quad \text{for all finite dimensional $X$}. \]

If the Lie group $\mathbb{R}^k$ acts on $X$, we define $d_\Omega(\cdot,\cdot)$
for any bounded set $\Omega\subset \mathbb{R}^k$ by (\ref{eq: definition of d_omega}).
$(X, d_\Omega)$ is homeomorphic to $(X, d)$, and we can define $\Widim_\varepsilon (X:\mathbb{R}^k)$
and the mean dimension $\dim(X:\mathbb{R}^k)$ by 
\begin{equation*}
 \begin{split}
 \Widim_\varepsilon (X:\mathbb{R}^k) &:= \lim_{n\to \infty} \frac{1}{n^k} 
\Widim_\varepsilon (X, d_{[0, n)^k}), \\
 \dim (X:\mathbb{R}^k) &:= \lim_{\varepsilon \downarrow 0} \Widim_\varepsilon (X:\mathbb{R}^k),
 \end{split}
\end{equation*}
where $n^k$ is the volume of $[0, n)^k$.
Here we have considered only $\mathbb{Z}^k$ and $\mathbb{R}^k$. But actually we can consider much more 
general groups; see Gromov \cite{Gromov}.
\begin{remark}
In the above definitions we have chosen special ``amenable sequences" $\{I_n\}_{n\geq 1}$ and $\{[0, n)^k\}_{n\geq 1}$
for simplicity of the explanation.
Actually the value of mean dimension does not depend on the choice of amenable sequences 
(this is a very important point). See Gromov \cite[pp. 335-338]{Gromov} and 
Lindenstrauss-Weiss \cite[Appendix]{Lindenstrauss-Weiss}.
\end{remark}
\begin{remark}
The above definition of mean dimension uses a distance. But actually mean dimension is a topological invariant;
if $d'$ is another distance on $X$ such that $(X, d')$ is homeomorphic to $(X,d)$, then we have
\[ \dim((X, d'):\mathbb{Z}^k) = \dim((X, d):\mathbb{Z}^k) .\]
This can be (easily) proved by using the fact: 
the identity map $id: (X, d)\to (X, d')$ becomes uniformly continuous (by the compactness of $X$).
See Gromov \cite[p. 339]{Gromov}.
\end{remark}
\begin{example}\label{example: fundamental example of mean dimension}
Let $\underline{X}$ be a compact metric space of finite covering dimension and set $X := \underline{X}^{\mathbb{Z}^k}$.
$\mathbb{Z}^k$ acts on $X$ by
\[ \mathbb{Z}^k \times X \to X, \quad 
(\gamma,\, (x_a)_{a\in \mathbb{Z}^k}) \mapsto \gamma. (x_a)_{a\in \mathbb{Z}^k}  
= (x_{\gamma + a})_{a\in \mathbb{Z}^k} \]
We define the distance $d(x, y)$ for $x=(x_a)_{a\in \mathbb{Z}^k}$ and $y= (y_a)_{a\in \mathbb{Z}^k}$ in $X$ by
\begin{equation}\label{eq: definition of distance in X}
 d(x, y) := \sum_{a\in \mathbb{Z}^k} 2^{-|a|} d(x_a, y_a) \quad 
\text{where $|a| = |a_1| +\cdots +|a_k|$ for $a =(a_1, \cdots, a_k)$} . 
\end{equation}
Then $X$ becomes a compact metric space. The mean dimension of $X$ is estimated by 
(see Lindenstrauss-Weiss \cite[Proposition 3.1]{Lindenstrauss-Weiss}):
\[ \dim(X:\mathbb{Z}^k) \leq \dim \underline{X} .\]
In addition, if $\underline{X}$ has a structure of finite polyhedron, then 
(see \cite[Proposition 3.3]{Lindenstrauss-Weiss})
\[ \dim(X:\mathbb{Z}^k) = \dim \underline{X} .\]
\end{example}
\begin{proof}
Let $n, s>0$ be positive integers and set $J:= (-s, n+s)^k \cap \mathbb{Z}^k$.
Let $\pi:X\to \underline{X}^J$ be the natural projection.
Some calculation shows that if $x , y\in X$ satisfy $\pi (x) = \pi (y)$, then 
\[ d_{I_n}(x, y) \leq C_{k, \underline{X}}\, 2^{-s},\]
where $C_{k, \underline{X}}$ is a positive constant depending on $k$ and $\Diam \underline{X}$.
For any $\varepsilon >0$, let $s$ be an integer satisfying $C_{k, \underline{X}}\, 2^{-s} <\varepsilon$.
(Note that we can take $s$ independent of $n$.)
Then if $\pi(x) = \pi (y)$, we have $d_{I_n}(x, y) <\varepsilon$.

The covering dimension of $\underline{X}^J$ is $\leq |J|\dim\underline{X}$.
Then for any $\delta>0$ there are a finite polyhedron $P$ of dimension $\leq |J|\dim\underline{X}$ and a
$\delta$-embedding $f:\underline{X}^J\to P$.
(Here we consider a distance on $\underline{X}^J$; the choice of the distance is not important.)
If we take $\delta$ sufficiently small, then the map $f\circ \pi :(X, d_{I_n})\to P$ becomes an $\varepsilon$-embedding.
Hence 
\[ \Widim_\varepsilon (X, d_{I_n}) \leq |J|\dim\underline{X} = (n+ 2s-1)^k\dim\underline{X}  .\]
Since $s$ is independent of $n$, we have
\[ \Widim_\varepsilon (X:\mathbb{Z}^k) = \lim_{n\to \infty} n^{-k} \Widim_\varepsilon (X, d_{I_n}) 
\leq \lim_{n\to \infty}(1 + \frac{2s-1}{n})^k\dim\underline{X} = \dim\underline{X}. \]
Therefore $\dim(X:\mathbb{Z}^k) \leq \dim\underline{X}$.

Next we suppose that $\underline{X}$ is a finite polyhedron of dimension $N$.
We want to prove $\dim(X:\mathbb{Z}^k) \geq N$.
There is a topological embedding $[0,1]^N \hookrightarrow \underline{X}$, 
and this induces a $\mathbb{Z}^k$-equivariant embedding $([0,1]^N)^{\mathbb{Z}^k} \hookrightarrow X$.
Hence it is enough to prove $\dim(([0,1]^N)^{\mathbb{Z}^k}:\mathbb{Z}^k) \geq N$.
Mean dimension is a topological invariant. 
Hence we can use any distance on $([0,1]^N)^{\mathbb{Z}^k}$. 
Here we consider the sup-distance $d_\infty$ on $[0,1]^N$ as in Example \ref{example: fundamental exmaple of Widim},
and we define the distance on $([0,1]^N)^{\mathbb{Z}^k}$ by (\ref{eq: definition of distance in X}):
\[ d(x, y) := \sum_{a\in \mathbb{Z}^k} 2^{-|a|} d_\infty (x_a, y_a). \]

Let $\iota: [0,1]^{N|I_n|} = ([0,1]^N)^{I_n} \hookrightarrow ([0,1]^N)^{\mathbb{Z}^k}$ 
be the embedding defined by
\[ \iota: (x_a)_{a\in I_n} \mapsto (y_a)_{a\in \mathbb{Z}^k} \quad 
\text{where $y_a = x_a \in [0,1]^N$ for $a\in I_n$ and $y_a =0$ for $a\not\in \mathbb{Z}^k$}. \]
By the definition of the distance $d_{I_n}$, we have
\[ d_\infty(x, y) \leq d_{I_n}( \iota(x), \iota(y) ) \quad \text{for $x, y\in [0,1]^{N|I_n|}$}.\]
Then, for any $\varepsilon <1$, 
\[ \Widim_\varepsilon (([0,1]^N)^{\mathbb{Z}^k}, d_{I_n}) \geq \Widim_\varepsilon ([0,1]^{N|I_n|}, d_\infty ) 
= N |I_n| .\]
Therefore
\[ \Widim_\varepsilon (([0,1]^N)^{\mathbb{Z}^k}:\mathbb{Z}^k) = 
\lim_{n\to \infty} n^{-n}\Widim_\varepsilon (([0,1]^N)^{\mathbb{Z}^k}, d_{I_n}) \geq N .\]
This shows
\[\dim(([0,1]^N)^{\mathbb{Z}^k}:\mathbb{Z}^k) \geq N .\]
\end{proof}
In the proof of Theorem \ref{theorem: upper bound for dim(M(X):C)} we used the following proposition 
(cf. Gromov \cite[p. 329]{Gromov} and Lindenstrauss-Weiss \cite[Proposition 2.7]{Lindenstrauss-Weiss}).
\begin{proposition}
Let $(X,d)$ be a compact metric space acted by the Lie group $\mathbb{R}^k$, and let $\Lambda \subset \mathbb{R}^k$
be a lattice. Then 
\[ \dim(X:\Lambda) = |\mathbb{R}^k/\Lambda|\dim(X:\mathbb{R}^k) ,\]
where $|\mathbb{R}^k/\Lambda|$ denotes the volume of the fundamental domain of $\Lambda$ in $\affine$.
\end{proposition}
\begin{proof}
We give the proof for the case of $\Lambda = \mathbb{Z}^k \subset \mathbb{R}^k$.
Other cases can be proved in the same way (by using different amenable sequences).

Since $d_{I_n}(\cdot, \cdot) \leq d_{[0, n)^k}(\cdot, \cdot)$, 
we have $\Widim_\varepsilon (X, d_{I_n}) \leq \Widim_\varepsilon (X, d_{[0, n)^k})$. Hence
\[ \dim(X:\mathbb{Z}^k )\leq \dim(X:\mathbb{R}^k) .\]
On the other hand, since the identity map $id: (X, d) \to (X, d_{[0, 1)^k})$ is uniformly continuous,
for any $\varepsilon >0$ there exists $\delta = \delta(\varepsilon) >0$ such that
\[ d(x, y) \leq \delta \Rightarrow d_{[0, 1)^k}(x, y)\leq \varepsilon \quad
\text{for any two $x, y\in X$}.\]
Note that 
\[ [0, n)^k = \bigsqcup_{u\in I_n} \{u+ [0, 1)^k\} .\]
Then 
\[ d_{I_n}(x, y) \leq \delta \Rightarrow d_{[0, n)^k}(x, y) \leq \varepsilon \quad
\text{for any two $x, y\in X$}.\]
Therefore
\[ \Widim_\epsilon (X,d_{[0, n)^k}) \leq \Widim_\delta (X, d_{I_n}) .\]
Hence 
\[ \Widim_\varepsilon (X:\mathbb{R}^k) \leq \Widim_\delta(X:\mathbb{Z}^k) \leq \dim(X:\mathbb{Z}^k). \]
Thus
\[ \dim(X:\mathbb{R}^k) \leq \dim(X:\mathbb{Z}^k) .\]
\end{proof}

\subsection{Some general results on mean dimension}
The proof of Theorem \ref{theorem: dim(M:C) = dim(M_+:C)} needs the following theorem.
\begin{theorem}\label{theorem: pre-fiber theorem}
Let $(X, d)$ and $(Y, d')$ be compact metric spaces acted by the additive group $\mathbb{Z}^k$, and
let $f:X\to Y$ be a $\mathbb{Z}^k$-equivariant continuous map.
Suppose that there exists a $\mathbb{Z}^k$-invariant closed subset $A$ in $X$ such that 
$f|_{X\setminus A}$ is injective. Then we have
\[ \dim(X:\mathbb{Z}^k) \leq \dim(Y:\mathbb{Z}^k) + \dim(A:\mathbb{Z}^k) .\]
\end{theorem}
\begin{proof}\footnote{The idea of this proof was inspired by the arguments in Robinson \cite[pp. 384-386]{Robinson}
 and Bowen \cite[Theorem 17]{Bowen}.}
Let $m$ be a positive integer and $\varepsilon$ be a positive number.
Let $i :(A, d_{I_m}) \to P$ be an $\varepsilon/2$-embedding with a 
$\Widim_{\varepsilon /2} (A, d_{I_m})$-dimensional finite polyhedron $P$.
Since a finite polyhedron is ANR (absolute neighborhood retract), there exist
a open set $U \supset A$ in $X$ and a continuous map $\tilde{i}:U\to P$ with $\tilde{i}|_A = i$.
Let $\rho: X \to [0,1]$ be a cut-off function on $X$ such that $\rho = 1$ on $A$ and $\mathrm{supp}(\rho) \subset U$.
Then we can define a continuous map $j$ from $X$ to the cone $C(K) := \{pt\} * K$ 
(the join of K and the one-point space $\{pt\}$) by $j(x) := (1-\rho(x))pt + \rho(x)\tilde{i}(x)$. 
Then we have the commutative diagram:
\begin{equation*}
 \begin{CD}
 A @>{i}>> K\\
 @VVV @VVV \\
 X @>{j}>> C(K)
 \end{CD}
\end{equation*}
We set $T := [0,\, \mathrm{Diam}(X, d)]\times C(K)$ and define the continuous map $g$ from $X$ to $T$ by 
\[ g: X \to T, \quad x \mapsto (d(x, A),\, j(x)).\]
Then the map $(f, g): (X, d_{I_m}) \to Y \times T$ becomes an $\varepsilon/2$-embedding because 
$f$ is injective on $X\setminus A$ and $g|_A :(A, d_{I_m}) \to T$ is an $\varepsilon/2$-embedding.
(Note that $g(A)$ and $g(X\setminus A)$ have no intersection: $g(A) \cap g(X\setminus A) = \emptyset$.)
Then there exists a positive number $\beta = \beta(m, \varepsilon)$ such that 
if two points $x_1$ and $x_2$ in $X$ satisfy $d'(f(x_1), f(x_2)) \leq \beta$ and $g(x_1) = g(x_2)$ then
$d_{I_m}(x_1, x_2) \leq \varepsilon$.

Let $n$ be a positive integer and define the positive integer $l$ by 
\begin{equation} \label{definition of l}
m(l-1) < n \leq ml.
\end{equation} 
We define the subset $\Gamma$ in $I_n$ by 
\[ \Gamma := \{( ma_1, ma_2, \cdots ,ma_k)\in m\mathbb{Z}^k|\, a_1, a_2, \cdots ,a_k \in \mathbb{Z}\, \text{ and } 
0 \leq a_1, a_2, \cdots, a_k\leq l-1 \} \]
From (\ref{definition of l}), we have 
\begin{equation}\label{partition of distance}
d_{I_n}(x_1, x_2) \leq \max_{\gamma\in \Gamma } d_{I_m}(\gamma.x_1,  \gamma.x_2 ) .
\end{equation}
Let $\pi :(Y, d'_{I_n}) \to Q$ be a $\beta$-embedding 
with a $\Widim_{\beta}(Y, d'_{I_n})$-dimensional finite polyhedron $Q$.
Define $\Pi :(X, d_{I_n}) \to Q \times T^\Gamma$ by 
\[ \Pi (x) := (\pi(f(x)),\, (g(\gamma.x ))_{\gamma\in \Gamma}) .\]
Suppose that two points $x_1$ and $x_2$ in $X$ satisfy $\Pi (x_1) = \Pi (x_2)$.
Then we have $g(\gamma. x_1) = g(\gamma. x_2)$ for all $\gamma\in \Gamma$
and $d'_{I_n}(f(x_1), f(x_2)) \leq \beta$. 
In particular, $d'(f(\gamma. x_1), f(\gamma. x_2))\leq \beta$ for all $\gamma \in \Gamma$.
From the definition of $\beta$, this implies $d_{I_m}(\gamma. x_1, \gamma. x_2) \leq \varepsilon$
for all $\gamma \in \Gamma$. 
Then (\ref{partition of distance}) shows
\[ d_{I_n}(x_1, x_2) \leq \varepsilon .\]
Thus $\Pi :(X, d_{I_n}) \to Q \times T^\Gamma$ is an $\varepsilon$-embedding.
The image space $Q \times T^\Gamma$ is a polyhedron. Hence
\begin{equation*}
 \begin{split}
 \frac{1}{n^k}\Widim_\varepsilon (X, d_{I_n}) &\leq \frac{1}{n^k}\dim(Q \times T^\Gamma ), \\
 & = \frac{1}{n^k}\Widim_\beta (Y, d'_{I_n}) + \frac{|\Gamma|}{n^k}\dim T,\\
 & \leq \frac{1}{n^k}\Widim_\beta (Y, d'_{I_n}) + (1/n + 1/m)^k \dim T.
\end{split}
\end{equation*}
Let $n$ go to infinity. Then we get
\begin{equation*}
 \begin{split}
 \Widim_\varepsilon (X:\mathbb{Z}^k) &\leq \Widim_\beta (Y: \mathbb{Z}^k) + m^{-k} \dim T, \\
 &\leq \dim(Y:\mathbb{Z}^k) + m^{-k} (\Widim_{\varepsilon/2} (A, d_{I_m}) + 2).
 \end{split}
\end{equation*}
Here we have used the fact: $\Widim_\beta (Y: \mathbb{Z}^k) \leq \dim(Y:\mathbb{Z}^k)$.
Let $m$ go to infinity. Then
\[ \Widim_\varepsilon (X:\mathbb{Z}^k) \leq \dim(Y:\mathbb{Z}^k) + \Widim_{\varepsilon/2} (A:\mathbb{Z}^k).\]
Let $\varepsilon$ go to $0$. Then we get the conclusion:
\[ \dim (X:\mathbb{Z}^k) \leq \dim (Y:\mathbb{Z}^k) + \dim (A:\mathbb{Z}^k) .\]
\end{proof}
\begin{remark}
For a general closed subset $A$ in $X$ (not necessarily $\mathbb{Z}^k$-invariant), we define 
$\dim (A:\{I_n\})$ by 
\[ \dim (A:\{I_n\}) := \lim_{\varepsilon\downarrow 0} 
\left(\liminf_{n\to \infty} \frac{1}{n^k} \Widim_\varepsilon (A, d_{I_n})\right).\]
(For the detail, see Gromov \cite[pp. 338-339]{Gromov}.)
Then the above proof shows the following result:
Let $(X, d)$ and $(Y,d')$ be compact metric spaces acted by $\mathbb{Z}^k$, and let $f:X\to Y$ be a 
$\mathbb{Z}^k$-equivariant continuous map. 
Suppose that there exists a closed subset $A$ in $X$ such that 
$f|_{X\setminus A}$ is injective. Then we have
\[ \dim(X:\mathbb{Z}^k) \leq \dim(Y:\mathbb{Z}^k) + \dim(A:\{I_n\}).\]
\end{remark}
\begin{problem}
I don't know whether the following statement is true or not (it might be too naive):
Let $(X, d)$ and $(Y,d')$ be compact metric spaces acted by $\mathbb{Z}^k$, and let $f:X\to Y$ be a 
$\mathbb{Z}^k$-equivariant continuous map. Then we have
\[ \dim(X:\mathbb{Z}^k) \leq \dim(Y:\mathbb{Z}^k) + \sup_{y\in Y}\{ \dim(f^{-1}(y):\{I_n \})\} .\]
\end{problem}
\begin{proof}[Proof of Theorem \ref{theorem: dim(M:C) = dim(M_+:C)}]
From $\moduli_+ \subset \moduli$, we have $\dim(\moduli_+:\affine ) \leq \dim(\moduli:\affine)$.
The reverse inequality is the problem.
Let $\Lambda\subset \affine$ be an arbitrary lattice, and we consider the discretization map
(cf. Section 2):
\[
D: \moduli \to (\affine P^N)^\Lambda , \quad f\mapsto f|_{\Lambda} .
\]
Since $e(f) = 0$ for all $f\in \moduli \setminus \moduli_+$, Lemma \ref{lemma: discretization 1} implies
that $D|_{\moduli\setminus \moduli_+}$ is injective.
Then we can apply Theorem \ref{theorem: pre-fiber theorem} to this situation, and we have
\[ \dim(\moduli:\Lambda ) \leq \dim(\moduli_+:\Lambda) + \dim((\affine P^N)^\Lambda :\Lambda) .\]
This means
\[ |\affine /\Lambda| \dim(\moduli:\affine) \leq |\affine /\Lambda| \dim(\moduli_+:\affine) + 2N.\]
We can let $|\affine /\Lambda|$ go to infinity. Hence
\[ \dim(\moduli:\affine) \leq  \dim(\moduli_+:\affine).\]
\end{proof}
Next proposition will be used in the proof of Theorem \ref{theorem: dim(M(X):C) = dim(M(Y):C)}.
\begin{proposition}\label{proposition: dim(X:Z^k) = dim(Y:Z^k)}
Let $X$ be a compact metric space acted by $\mathbb{Z}^k$ and $Y\subset X$ 
be a $\mathbb{Z}^k$-invariant closed subset in $X$.
Suppose that the complement $Y^c = X\setminus Y$ has a finite covering dimension. Then
\[ \dim(X:\mathbb{Z}^k) = \dim(Y:\mathbb{Z}^k) .\]
\end{proposition}
\begin{proof}
This proposition is a corollary of Theorem \ref{theorem: pre-fiber theorem}.
If $Y=X$, then the statement is trivial. Hence we suppose $Y\neq X$.
We define $X/Y$ by
\[ X/Y := X/\sim \quad \text{where $y_1\sim y_2$ for all $y_1, y_2\in Y$}.\]
We give the quotient topology to $X/Y$.
It is easy to see that $X/Y$ becomes a second countable compact Hausdorff space.
Hence we can give a distance $d(\cdot, \cdot)$ to $X/Y$ (by Urysohn's theorem).
In addition $X/Y$ is finite dimensional. In fact
\[ X/Y = \bigcup_{n\geq 1}\{x\in X/Y|\, d(x, [Y]) \geq 1/n\} \cup \{[Y]\} \]
where $[Y]$ is the point in $X/Y$ corresponding to $Y\subset X$.
The set $\{d(x, [Y])\geq 1/n\}$ is homeomorphic to a closed subset in $Y^c$, 
and hence its dimension is $\leq \dim(Y^c)$.
Thus 
\[ \dim(X/Y) = \max_{n\geq 1}(\dim \{d(x, [Y])\geq 1/n\} , \dim\{[Y]\}) \leq \dim (Y^c).\]

Since $Y$ is $\mathbb{Z}^k$-invariant, $\mathbb{Z}^k$ naturally acts on $X/Y$ and the projection 
$\pi: X\to X/Y$ becomes $\mathbb{Z}^k$-equivariant. The finite dimensionality of $X/Y$ implies
$\dim(X/Y:\mathbb{Z}^k) =0$.
$\pi|_{Y^c}$ is injective. 
Then we can apply Theorem \ref{theorem: pre-fiber theorem} and get
\[\dim(X:\mathbb{Z}^k) \leq \dim(X/Y:\mathbb{Z}^k) + \dim(Y:\mathbb{Z}^k) = \dim(Y:\mathbb{Z}^k).\]
On the other hand the reverse inequality $\dim(Y:\mathbb{Z}^k) \leq \dim(X:\mathbb{Z}^k)$ is trivial.
\end{proof}
Let $\underline{X}, \underline{Y}$ be compact metric spaces of finite covering dimension and set
$X:=\underline{X}^{\mathbb{Z}^k}, Y:= \underline{Y}^{\mathbb{Z}^k}$.
The additive group $\mathbb{Z}^k$ acts on $X$ and $Y$ as in Example \ref{example: fundamental example of mean dimension}.
Let $f: \underline{X}\to \underline{Y}$ be a continuous map. 
We define a $\mathbb{Z}^k$-equivariant continuous map $F:X\to Y$ by
\[ F: (x_a)_{a\in \mathbb{Z}^k} \mapsto (f(x_a))_{a\in \mathbb{Z}^k} .\]
Let $\Delta \subset Y$ be the diagonal ($\Delta \cong \underline{Y}$), and set $Z:= F^{-1}(\Delta) \subset X$.
(This is an easy example of ``subshifts of finite type" in Gromov \cite[p. 324]{Gromov}.)
The following will be used in the proof of Theorem \ref{theorem: holomorphic 1-forms and mean dimension}.
\begin{proposition}\label{proposition: mean dimension of subshift}
\[ \dim(Z:\mathbb{Z}^k) \leq \sup_{\underline{y}\in\underline{Y}}\, (\dim f^{-1}(\underline{y})) .\]
\end{proposition}
\begin{proof}
We will use the same notations as in the proof of Example \ref{example: fundamental example of mean dimension}:
Let $s, n$ be positive integers and set $J:=(-s, n+s)^k \cap \mathbb{Z}^k$.
Let $\pi :X\to \underline{X}^J$ be the natural projection.
For any $\varepsilon >0$, there exists $s = s(\varepsilon , k, \underline{X})$ such that 
\[ \Diam (\pi^{-1}(p), d_{I_n}) < \varepsilon \quad \text{for all $p \in \underline{X}^J$ and any 
$n>0$}.\]
This implies
\[ \Widim_\varepsilon (Z, d_{I_n}) \leq \dim \pi(Z).\]
By the definition of $Z$, there is a (unique) continuous map $g:\pi(Z) \to \underline{Y}$ such that the following diagram becomes commutative:
\begin{equation*}
 \begin{CD}
 Z @>{F}>> \Delta \\
 @V{\pi}VV @V{\cong}VV \\
 \pi(Z) @>{g}>> \underline{Y}
 \end{CD}
\end{equation*}
For each $\underline{y}\in \underline{Y}$, we have $g^{-1}(\underline{y}) \subset (f^{-1}(\underline{y}))^J$.
Then the topological dimension theory gives
\begin{equation*}
 \begin{split}
 \dim \pi (Z) &\leq \dim\underline{Y} + \sup_{\underline{y}\in \underline{Y}}\, 
(\dim g^{-1}(\underline{y})), \\
              &\leq \dim\underline{Y} + |J| \sup_{\underline{y}\in \underline{Y}}\, 
(\dim f^{-1}(\underline{y})).
 \end{split}
\end{equation*}
Therefore
\[ n^{-k} \Widim_\varepsilon (Z, d_{I_n}) \leq 
n^{-k}\dim\underline{Y} + (1+ \frac{2s-1}{n})^k \sup_{\underline{y}\in \underline{Y}}\, (\dim f^{-1}(\underline{y})).\]
Let $n$ go to infinity. Then we get
\[ \Widim_\varepsilon (Z:\mathbb{Z}^k) \leq \sup_{\underline{y}\in \underline{Y}}\, (\dim f^{-1}(\underline{y})).\]
Thus
\[ \dim(Z:\mathbb{Z}^k) \leq \sup_{\underline{y}\in \underline{Y}}\, (\dim f^{-1}(\underline{y})).\]
\end{proof}


\section{Holomorphic 1-forms and mean dimension}
\subsection{Proof of Theorem \ref{theorem: dim(M(X):C) = dim(M(Y):C)}}
Let $X$ be a compact K\"{a}hler manifold\footnote{In the proof of Theorem \ref{theorem: dim(M(X):C) = dim(M(Y):C)}
we don't use the results in Section 2.} and let $\alpha: X\to \Alb(X)$ be the 
Albanese map. Set
\[ Y := \{x\in X|\, \text{$d\alpha_x:T_xX \to T_{\alpha(x)}\Alb(X)$ is not injective}\}.\]
$Y$ is a closed analytic set in $X$.
We define $\moduli (X)$ and $\moduli (Y)$ by 
\begin{equation*}
 \begin{split}
 \moduli (X)  &:= \{f:\affine \to X:\text{holomorphic}|\, |df|\leq 1\}, \\
 \moduli (Y)  &:= \{f\in \moduli (X)|\, f(\affine) \subset Y\} .
 \end{split}
\end{equation*}
Here we define $|df|\geq 0$ by using the K\"{a}hler form $\omega$ on $X$ and 
the equation (\ref{eq: definition of |df|}):
\[ f^*\omega = |df|^2(z) dxdy .\]

We want to prove 
\begin{equation}\label{eq: dim(M(X):C) = dim(M(Y):C)}
\dim(\moduli (X):\affine) = \dim(\moduli (Y):\affine).
\end{equation}
Actually we will prove the following lemma.
\begin{lemma}\label{lemma: key for the proof of dim(M(X):C) = dim(M(Y):C)}
For any $f\in \moduli (X)\setminus \moduli (Y)$, there exists a closed neighborhood 
$K\subset \moduli (X)\setminus \moduli (Y)$ of $f$ such that $\dim K \leq 4\dim_\affine X$.
\end{lemma}
If this lemma is proved, we can prove (\ref{eq: dim(M(X):C) = dim(M(Y):C)}) as follows;
Since $\moduli (X)\setminus \moduli (Y)$ is $\sigma$-compact (i.e. a union of countable compact sets),
$\moduli (X)\setminus \moduli (Y)$ becomes a union of countable closed sets of dimension $\leq 4\dim_\affine X$:
\[ \moduli (X)\setminus \moduli (Y) = \bigcup_{n\geq 1} K_n, \quad \text{$K_n$: closed and $\dim K_n \leq 4\dim_\affine X$}.\]
Hence
\[ \dim(\moduli (X) \setminus \moduli (Y)) = \sup_{n\geq 1}\, (\dim K_n) \leq 4\dim_\affine X .\]
Then we can apply Proposition \ref{proposition: dim(X:Z^k) = dim(Y:Z^k)} and get
\[ \dim(\moduli (X):\affine) = \dim(\moduli (X):\mathbb{Z}^2) = \dim(\moduli (Y):\mathbb{Z}^2) = \dim(\moduli (Y):\affine).\]
The proof of Lemma \ref{lemma: key for the proof of dim(M(X):C) = dim(M(Y):C)} uses the following obvious fact:
\begin{lemma} \label{lemma: brody curves in complex torus}
Let $T=\affine^h/\Gamma$ be a complex torus with a Hermitian metric.
Let $f:\affine \to T$ be a holomorphic curve satisfying 
\[ \norm{df} := \sup_{z\in \affine}|df|(z) <\infty .\]
(Here we define $|df|(z)$ by the equation (\ref{eq: definition of |df|}).)
Then $f$ can be expressed by 
\[ f(z) = [Az+B] \quad \text{where $A, B \in \affine^h$} .\]
In particular, let $f, g:\affine \to T$ be holomorphic curves satisfying 
$\norm{df}, \norm{dg} <\infty$. 
If $df(\partial /\partial z)|_{z=0} = dg (\partial /\partial z)|_{z=0}$ (in particular $f(0) =g(0)$), then 
$f\equiv g$.
\end{lemma}
\begin{proof}[Proof of Lemma \ref{lemma: key for the proof of dim(M(X):C) = dim(M(Y):C)}]
We can suppose $Y \not =X$.
Since $f\not \in \moduli (Y)$, we have $f(\affine )\not \subset Y$. 
We suppose $f(0) \not\in Y$ for simplicity. 
Let $D \subset X\setminus Y$ be a compact neighborhood of $f(0)$. 
We define a closed neighborhood $K$ of $f$ by
\[ K:= \{g\in\moduli (X)|\, g(0) \in D \} \subset \moduli (X)\setminus \moduli (Y) .\]
$K$ is closed in $\moduli (X)$ (hence $K$ is compact), and $f$ is an interior point of $K$.
Consider the following continuous map:
\[ S: K\to TX, \quad g\mapsto dg(\partial /\partial z)|_{z=0} .\]
$S$ is injective; if $S(g_1) = S(g_2)$, 
then we have $d(\alpha\circ g_1)(\partial /\partial z)|_{z=0} = d(\alpha\circ g_2)(\partial /\partial z)|_{z=0}$
(here $\alpha:X\to \Alb(X)$ is the Albanese map). 
Lemma \ref{lemma: brody curves in complex torus} implies $\alpha\circ g_1 \equiv \alpha\circ g_2$.
The Albanese map $\alpha$ is a local embedding in a neighborhood of $g_1(0)=g_2(0) \in X\setminus Y$.
Therefore $g_1(z) = g_2(z)$ if $|z| \ll 1$. 
From the unique continuation principle, we have $g_1\equiv g_2$.
Hence $S$ is injective. Since $K$ is compact, $S$ is a homeomorphism from $K$ to $S(K) \subset TX$.
Thus 
\[\dim K = \dim S(K) \leq \dim TX = 4\dim_\affine X .\]
\end{proof}


\subsection{Proof of Theorem \ref{theorem: holomorphic 1-forms and mean dimension}}
The proof of Theorem \ref{theorem: holomorphic 1-forms and mean dimension} is based on the following fact:
A bounded holomorphic 1-form on the complex plane $\affine$ is of the form 
\[ a dz \quad \text{where $a$ is a constant}. \] 

Let $X$ be a smooth, connected projective variety, and let $\omega_1, \cdots, \omega_h$ be 
a basis of $H^{1,0}$ ($h= \dim_\affine H^{1,0}$). 
Let $d\alpha:TX\to \affine^h$ be the derivative of the Albanese map $\alpha$:
\[ d\alpha :TX\to \affine^h , \quad v\mapsto (\omega_1(v), \cdots, \omega_h(v) ).\]
Let $BX$ be the ball bundle:
\[ BX:= \{v\in TX|\, |v|\leq 1 \} .\]
Let $D := \{u\in \affine^h|\, |u|\leq R\}$ be the ball of radius $R$ in $\affine^h$. 
Here we take $R$ sufficiently large so that $d\alpha (BX) \subset D$.

Consider a lattice $\Lambda \subset \affine$ satisfying 
\[ e(X) < \frac{1}{|\affine/\Lambda|} .\]
Then Lemma \ref{lemma: discretization 2} implies that the following discretization map $S$ is a topological 
embedding:
\[S: \moduli (X) \to BX^\Lambda ,
\quad f\mapsto (df(\partial /\partial z)|_{z=\lambda})_{\lambda\in \Lambda} .\]
(Note that $|df(\partial /\partial z)| = |df|(z)/\sqrt{2} \leq 1/\sqrt{2} <1$.)
The map $d\alpha|_{BX}:BX\to D$ defines 
\[ A: BX^\Lambda \to D^\Lambda , 
\quad (u_{\lambda})_{\lambda\in \Lambda} \mapsto (d\alpha(u_\lambda))_{\lambda\in \Lambda}.\]
Let $\Delta\subset D^\Lambda$ be the diagonal. 
Then Proposition \ref{proposition: mean dimension of subshift} shows
\[ \dim(A^{-1}(\Delta): \Lambda) \leq \sup_{u\in D} \dim\{(d\alpha)^{-1}(u)\cap BX\} 
\leq 2\sup_{u\in \affine^h} \dim_\affine d\alpha^{-1}(u) .\]

For any $f \in \moduli (X)$, $f^*\omega_i$ is a bounded holomorphic 1-form on $\affine$.
Hence it is of the form $adz$ ($a$ is a constant depending on $f$ and $\omega_i$).
This means that $S(\moduli (X))$ is contained in $A^{-1}(\Delta)$. 
Therefore 
\[ \dim(\moduli (X):\Lambda) = \dim(S(\moduli (X)):\Lambda) \leq \dim(A^{-1}(\Delta): \Lambda) \leq
2\sup_{u\in \affine^h} \dim_\affine d\alpha^{-1}(u) .\]
From this, we have
\[ \dim(\moduli (X):\affine) = \frac{\dim(\moduli (X):\Lambda)}{|\affine /\Lambda|} 
\leq \frac{2}{|\affine /\Lambda|}  \sup_{u\in \affine^h} \dim_\affine d\alpha^{-1}(u) .\]
We can take $1/|\affine /\Lambda|$ arbitrarily close to $e(X)$. Thus we conclude that 
\[ \dim(\moduli (X):\affine) \leq 2e(X) \sup_{u\in \affine^h} \dim_\affine d\alpha^{-1}(u) .\]


\vspace{10mm}

\address{ Masaki Tsukamoto \endgraf
Department of Mathematics, Faculty of Science \endgraf
Kyoto University \endgraf
Kyoto 606-8502 \endgraf
Japan
}

\textit{E-mail address}: \texttt{tukamoto@math.kyoto-u.ac.jp}

\end{document}